\documentclass[11pt]{amsart}

\usepackage[english]{babel}
\usepackage{amsmath}
\usepackage{amsfonts}
\usepackage{amssymb}
\usepackage{amsthm}
\usepackage{epsfig}
\usepackage{graphicx}
\usepackage{url,hyperref}

\newtheorem{theorem}{Theorem}[section]
\newtheorem{lemma}{Lemma}[section]
\newtheorem{corollary}{Corollary}[section]
\newtheorem{proposition}{Proposition}[section]

\newtheorem{conjecture}{Conjecture}[section]
\newtheorem{remark}{Remark}[section]

\DeclareMathOperator{\inter}{int}
\def\lin{\mathop\mathrm{lin}\nolimits}
\def\aff{\mathop\mathrm{aff}\nolimits}
\def\conv{\mathop\mathrm{conv}\nolimits}
\def\bd{\mathop\mathrm{bd}\nolimits}

\def\K{\mathcal{K}}

\def\r{\mathcal{R}}
\def\R{\mathbb{R}}
\def\N{\mathbb{N}}

\def\vol{\mathrm{vol}}

\def\C{\mathbb{C}}

\def\W{\mathrm{W}}

\def\e{\mathrm{e}}
\def\im{\mathrm{i}}
\def\sy#1#2{\sigma_{#1}\left(#2\right)}

\def\Real{\mathop\mathrm{Re}\nolimits}
\def\Im{\mathop\mathrm{Im}\nolimits}

\def\f#1#2{f_{#1;#2}}

\newcommand{\dotcup}{\ensuremath{\mathaccent\cdot\cup}}
\newcommand{\bigdotcup}[1][0pt]{\mathaccent\cdot{}\kern-#1\bigcup}

\numberwithin{equation}{section}

\begin{document}

\title{Steiner polynomials via ultra-logconcave sequences}

\author{Martin Henk}
\address{Fakult\"at f\"ur Mathematik, Otto-von-Guericke
Universit\"at Mag\-deburg, Universit\"atsplatz 2, D-39106-Magdeburg,
Germany} \email{martin.henk@ovgu.de; eugenia.saorin@ovgu.de}

\author{Mar\'\i a A. Hern\'andez Cifre}
\address{Departamento de Matem\'aticas, Universidad de Murcia, Campus de
Espinar\-do, 30100-Murcia, Spain} \email{mhcifre@um.es}

\author{Eugenia Saor\'\i n}

\thanks{Supported by MCI and DAAD, Ref. AIB2010DE-00082. Second author is
supported by Subdirecci\'on General de Proyectos de Investigaci\'on (MCI)
MTM2009-10418 and by ``Programa de Ayudas a Grupos de Excelencia de la
Regi\'on de Murcia'', Fundaci\'on S\'eneca, 04540/GERM/06.}



\begin{abstract}
We investigate structural properties of the cone of roots of relative
Steiner polynomials of convex bodies. We prove that they are closed,
monotonous with respect to the dimension, and that they cover the whole
upper half-plane, except the positive real axis, when the dimension tends
to infinity. In particular, it turns out that relative Steiner polynomials
are stable polynomials if and only if the dimension is $\leq 9$. Moreover,
pairs of convex bodies whose relative Steiner polynomial has a complex
root on the boundary of such a cone have to satisfy some
Aleksandrov-Fenchel inequality with equality. An essential tool for the
proofs of the results is the characterization of Steiner polynomials via
ultra-logconcave sequences.
\end{abstract}

\maketitle

\section{Introduction}

Let $\K^n$ be the set of all convex bodies, i.e., compact convex sets, in
the $n$-dimensional Euclidean  space $\R^n$, and let $B_n$ be the
$n$-dimensional unit ball. The subset of $\K^n$ consisting of all convex
bodies with non-empty interior is denoted by $\K_0^n$. The volume of a set
$M\varsubsetneq\R^n$, i.e., its $n$-dimensional Lebesgue measure, is
denoted by $\vol(M)$, its boundary by $\bd M$ and its convex, affine and
linear hulls by $\conv M$, $\aff M$ and $\lin M$, respectively. For two
convex bodies $K,E\in\K^n$ and a non-negative real number $\lambda$, the
volume of the Minkowski sum $K+\lambda\,E$ is expressed as a polynomial of
degree at most $n$ in $\lambda$, and it is written as
\begin{equation}\label{eq:steiner-minkowski}
\vol(K+\lambda E)=\sum_{i=0}^n\binom{n}{i}\W_i(K;E)\lambda^i.
\end{equation}
This expression is called {\em Minkowski-Steiner formula} or {\em relative
Steiner formula} of $K$. The coefficients $\W_i(K;E)$ are the {\em
relative quermassintegrals} of $K$, and they are a special case of the
more general defined {\em mixed volumes} for which we refer to
\cite[s.~5.1]{Sch}. In particular, we have $\W_0(K;E)=\vol(K)$,
$\W_n(K;E)=\vol(E)$, $\W_i(\mu_1\,K;\mu_2\,E)=\mu_1^{n-i}\mu_2^i\W_i(K;E)$
for $\mu_1,\mu_2\geq 0$ and $\W_i(K;E)=\W_{n-i}(E;K)$.

In the following we regard the right hand side in
\eqref{eq:steiner-minkowski} as a formal polynomial in a complex variable
$z\in\C$, which we will denote by
\[
\f{K}{E}(z)=\sum_{i=0}^n\binom{n}{i}\W_i(K;E)z^i.
\]
It is known that $\W_i(K;E)\geq 0$, with equality if and only if $\dim
K<n-i$ or $\dim E<i$ (see e.g. \cite[Theorem~5.1.7]{Sch}). Hence, with
respect to the dimensions of the bodies $K,E$ we may write
\begin{equation*}
\f{K}{E}(z)=\sum_{i=n-\dim K}^{\dim E}\binom{n}{i}\W_i(K;E)z^i.
\end{equation*}
Moreover, since $\W_i(K;E)=\W_{n-i}(E;K)$ we have
$\f{K}{E}(z)=z^n\,\f{E}{K}(1/z)$, and thus, up to multiplication by real
constants,
\begin{equation}\label{e:f_KE_f_EK}
\f{K}{E}(z)\text{ and }\f{E}{K}(z)\text{ have the same non-trivial roots.}
\end{equation}
Here we are interested in the location of the roots of $\f{K}{E}(z)$. To
this end, let $\C^+=\{z\in\C:\Im(z)\geq 0\}$ be the set of complex numbers
with non-negative imaginary part, and we denote by $\R_{\leq 0}$ and
$\R_{>0}$ the non-positive and positive real axes, respectively. For any
dimension $n\geq 2$, let
\begin{equation}\label{e:R_n}
\r(n)=\bigl\{z\in\C^+: \f{K}{E}(z)=0 \text{ for }K,E\in\K^n,\,
\dim(K+E)=n\bigr\}
\end{equation}
be the set of all roots of all non-trivial Steiner polynomials in the
upper half-plane. Note,  that if $\dim(K+E)<n$ then all relative
quermassintegrals vanish and so $\f{K}{E}(z)\equiv 0$.

By the isoperimetric inequality for arbitrary gauge bodies $E$ (cf.~e.g.,
\cite[p.~317-318]{Sch}), it is easy to see that $\r(2)=\R_{\leq 0}$ is
exactly the non-positive real axis and, in particular, it is a convex
cone. For arbitrary dimensions this was verified in \cite{HHC2}. More
precisely, the following result was shown.
\begin{theorem}[{\cite[Theorem 1.1]{HHC2}}]\label{t:cone}
$\r(n)$ is a convex cone containing $\R_{\leq 0}$.
\end{theorem}

Hence one ray of the boundary of $\r(n)$ consists of the non-positive real
axis $\R_{\leq 0}$, and, of course,  any odd-degree Steiner polynomial has
a root on this boundary. The ``other ray'' of the boundary of $\r(n)$
seems to have more geometric structure. We  call a pair of convex bodies
$(K,E)\in\K^n\times\K^n$ a {\it boundary-pair} if the Steiner polynomial
$\f{K}{E}(z)$ has a root on the boundary $\bd\r(n)\backslash\R_{\leq 0}$,
and in view of \eqref{e:f_KE_f_EK} we may additionally assume $\dim
K\leq\dim E$.

Regarding the $3$-dimensional case, in \cite{HHC2} the following
characterization was given.
\begin{proposition}[{\cite[Theorem 1.2]{HHC2}}]\label{p:cone3}
$\r(3)=\left\{x+y\im\in\C^+: x+\sqrt{3}\,y\leq 0\right\}$. Moreover, a
pair $(K,E)$ is a boundary-pair if and only if $\dim K=2$, $\dim E=3$ and
$\W_2(K;E)^2=\W_1(K;E)\W_3(K;E)$.
\end{proposition}
We notice that two convex bodies $K,E\in\K^n$ satisfy the above conditions
if and only if $E\in\K^3_0$ is a cap-body of (an homothet of) a planar
convex body $K$ (see \cite{Bol}). A convex body $L$ is called a cap-body
of $M\in\K^n$ if $L$ is the convex hull of $M$ and countably many points
such that the line segment joining any pair of these points intersects
$M$.

Here we will also extend the exact description of the cones $\r(n)$ to the
case $n=4$, and get similarly to $n=3$ the following characterization.
\begin{proposition}\label{p:r(4)}
$\r(4)=\left\{x+y\im\in\C^+: x+y\leq 0\right\}$. Moreover, a pair $(K,E)$
is a boundary-pair if and only if $\dim K=3$, $\dim E=4$ and, for $i=2,3$,
$\W_i(K;E)^2=\W_{i-1}(K;E)\W_{i+1}(K;E)$.
\end{proposition}
However, in contrast to the case $n=3$ we are not aware of an equivalent
geometric description of the boundary pairs $(K,E)$ in dimension $4$.

The cones $\r(2),\r(3),\r(4)$ are in particular closed, and our first main
result verifies this in any dimension.
\begin{theorem}\label{t:R_n_closed}
The cone $\r(n)$ is closed.
\end{theorem}

The low dimensional cones are also strictly nested, i.e.,
$\r(2)\varsubsetneq\r(3)\varsubsetneq\r(4)$. Our second theorem shows that
this is also true in general.
\begin{theorem}\label{t:R_n_inclusion}
$\r(n)\varsubsetneq\r(n+1)$.
\end{theorem}

So, the following question arises in a natural way: does $\r(n)$ cover the
whole upper half-plane $\C^+$, except $\R_{>0}$, when $n$ tends to
infinity? Next theorem  gives an affirmative answer to it.
\begin{theorem}\label{t:n->infty}
Let $\gamma\in\C^+\setminus\R_{>0}$. Then there exists $n_{\gamma}\in\N$
with $\gamma\in\r(n)$ for all $n\geq n_{\gamma}$.
\end{theorem}

It is well-known that the relative quermassintegrals of two convex bodies
satisfy the inequalities
\begin{equation}\label{e:special_af2}
\W_i(K;E)^2\geq\W_{i-1}(K;E)\W_{i+1}(K;E),\;\; 1\leq i\leq n-1,
\end{equation}
which are particular cases of the Aleksandrov-Fenchel inequality; we
notice that the complete classification of the equality cases is an
unsolved problem (see e.g.~\cite[ss.~6.3, 6.6]{Sch}). By the proof of
Theorem \ref{t:R_n_inclusion} the following corollary is obtained, which
also shows that boundary-pairs have a special geometric meaning (cf.
Propositions \ref{p:cone3} and \ref{p:r(4)}).
\begin{corollary}\label{c:A-F_ineq}
For $n\geq 3$, let $(K,E)$ be a boundary-pair. Then there exists
$i\in\{1,\dots,n-1\}$ such that
\begin{equation}\label{e:equality_A-F}
\W_i(K;E)^2=\W_{i-1}(K;E)\W_{i+1}(K;E),
\end{equation}
i.e., $K,E$ are extremal sets for at least one Aleksandrov-Fenchel
inequality.
\end{corollary}

According to Propositions \ref{p:cone3} and \ref{p:r(4)} all Steiner
polynomials for $n=3,4$ of boundary-pairs are (up to multiplication by a
constant) of the type
\begin{equation*}
\sum_{i=1}^3\binom{3}{i}\lambda^{3-i}z^i\quad \text{ and }\quad
\sum_{i=1}^4\binom{4}{i}\lambda^{4-i}z^i,
\end{equation*}
for all real $\lambda\geq 0$. Since the parameter $\lambda$ implies just a
multiplication of the roots and $\r(n)$ is a convex cone, we can say that
a representative of the Steiner polynomials of boundary-pairs is given by
a truncated binomial polynomial (setting $\lambda=1$) for $n=3,4$. We
believe that this is true in general, and so for $0\leq j<k\leq n$ we
define
\[
P_{j,k}^n(z):=\sum_{i=j}^k\binom{n}{i}z^i,
\]
the truncation of the binomial polynomial $(z+1)^n$ with indices
$j<k$.

\begin{conjecture}\label{co:binomial}
Let $n\geq 5$ and let $\gamma\in\bd\r(n)\backslash\R_{\leq 0}$. Then there
exist a truncated binomial polynomial $P_{j,k}^n(z)$, $0<j<k<n$, and
$\lambda>0$, such that $P_{j,k}^n(\lambda\gamma)=0$.
\end{conjecture}
Notice that the conjecture would directly imply that if $(K,E)$ is a
bound\-ary-pair, then $K,E$ are extremal sets for exactly $n-3$
Aleksandrov-Fenchel inequalities (cf. Corollary \ref{c:A-F_ineq}).

\medskip

The property that all roots of $3$-dimensional Steiner polynomials lie in
the left half-plane was part of a conjecture posed by Sangwine-Yager
\cite{SY88n} (cf.~e.g., \cite[p.\;65]{SY93}), motivated by a problem of
Teissier \cite{Te82}. There it was claimed that Steiner polynomials
satisfy $\r(n)\subseteq\bigl\{z\in \C^+:\Real(z)\leq 0\bigr\}$. This
inclusion is known to be true for dimensions $\leq 9$. In fact, in
\cite[Proposi\-tion~1.1]{HHC2} it was shown that
\begin{equation}\label{e:stability}
\r(n)\subseteq\bigl\{z\in\C^+:\Real(z)<0\bigr\}\cup\{0\}\quad\text{ for
}\; n\leq 9,
\end{equation}
i.e., all non-trivial roots are in the open left half-plane. We will call
this property ``weak'' stability. In \cite{HHC} the conjecture was shown
to be false in dimensions $\geq 12$ for a special family of bodies (see
also \cite{Kat09} for another family of high dimensional convex bodies
with this property). By looking at the roots of particular truncated
polynomials, we get rid of the gap, showing that for $n=10,11$ Steiner
polynomials are also not weakly stable.

\begin{proposition}\label{p:cones_10_11}
Steiner polynomials are weakly stable polynomials, i.e.,
$\r(n)\subseteq\bigl\{z\in\C^+:\Real(z)<0\bigr\}\cup\{0\}$, if and only if
$n\leq 9$.
\end{proposition}

Figure \ref{f:roots_cones} depicts the above results, and
for further information on the roots of Steiner polynomials in the context
of Teissier's problem  we refer to \cite{HHC,HCS,HCS2,Jet,Kat09}.

\begin{figure}[h]
\begin{center}
\includegraphics[width=8.6cm]{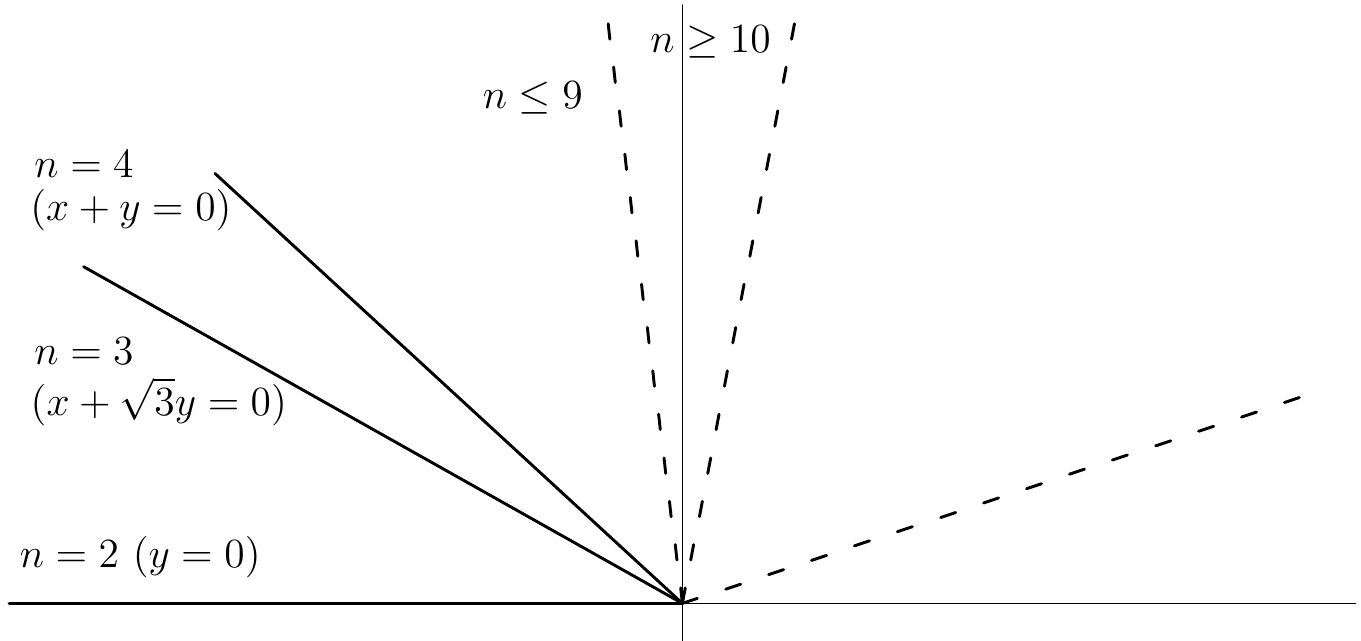}
\caption{Structure of the cones $\r(n)$.}\label{f:roots_cones}
\end{center}
\end{figure}

An essential tool for the proofs of the above results is a
characterization of Steiner polynomials via ultra-logconcave sequences
(Lemma \ref{l:charact_Steiner}), which we show in
Section~\ref{s:preliminary}. There we also discuss some additional
applications of this characterization. In Sections~\ref{s:boundary} and
\ref{s:monotonicity} we present the proofs of the main results, namely
Theorems~\ref{t:R_n_closed}, \ref{t:R_n_inclusion} and \ref{t:n->infty},
as well as some consequences. Finally, in Section~\ref{s:apendix} we
characterize the $4$-dimensional cone $\r(4)$.

\section{Ultra-logconcave sequences}\label{s:preliminary}

A sequence of non-negative real numbers $a_0,\dots,a_n$ is said to be {\it
ultra-logconcave} if
\begin{equation}\label{e:ultra-log}
c_{i,n}\, a_i^2 \geq a_{i-1}\,a_{i+1}\quad\text{ with }\;
c_{i,n}=\frac{\binom{n}{i-1}\binom{n}{i+1}}{\binom{n}{i}^2}=\frac{i}{i+1}\frac{n-i}{n-i+1},
\end{equation}
$1\leq i\leq n-1$. For further information on ultra-logconcave sequences
we refer to \cite{Gu,Li} and the references inside. This property for real
numbers allows to characterize Steiner polynomials.
\begin{lemma}\label{l:charact_Steiner}
A real polynomial $\sum_{i=0}^na_iz^i$, $a_i\geq 0$, is a Steiner
polynomial $\f{K}{E}(z)$ for a pair of convex bodies $K,E\in\K^n$, with
$\dim E=r$, $\dim K=s$, $\dim(K+E)=n$, if and only if
\begin{enumerate}
\item[i)] $a_i>0$ for all $n-s\leq i\leq r$, and $a_i=0$ otherwise, and
\item[ii)] the sequence $a_0,\dots,a_n$ is ultra-logconcave, i.e.,
\[
c_{i,n}\, a_i^2 \geq a_{i-1}\,a_{i+1}\;\text{ for }1\leq i\leq n-1.
\]
\end{enumerate}
\end{lemma}

This result essentially follows from a theorem of Shephard
(\cite[Theorem~4]{She}, see also \cite[p.~333]{Sch}, and for the
$2$-dimensional case see \cite{He}), which states that any given set of
$n+1$ non-negative real numbers $\W_0,\dots,\W_n\geq 0$ satisfying the
inequalities $\W_i\W_j\geq\W_{i-1}\W_{j+1}$, $1\leq i\leq j\leq n-1$,
arises as the set of relative quermassintegrals of two convex bodies.
There, an explicit construction of the two convex bodies is given in the
case when all $\W_i>0$, whereas the general case is obtained by a rather
non-constructive topological argument. Here we reduce the number of
involved inequalities, and extend the construction of the two convex
bodies to $\W_i\geq 0$.

\begin{proof}[Proof of Lemma \ref{l:charact_Steiner}]
If $\sum_{i=0}^na_iz^i$ is the Steiner polynomial of two convex bodies
$K,E\in\K^n$, then $a_i=\binom{n}{i}\W_i(K;E)$, and i), ii) are well-known
properties of quermassintegrals. For i) see \cite[Theorem 5.1.7]{Sch} and
ii) is \eqref{e:special_af2}.

Now we assume i) and ii). If $s=0$ or $r=0$ then both $a_nz^n$ and $a_0$
are obviously Steiner polynomials, and so we may assume $r,s\geq 1$.
Setting $\W_i=a_i/\binom{n}{i}$, we directly get that
\begin{equation}\label{e:cond_Wi}
\begin{split}
 & \W_i>0 \text{ for all } n-s\leq i\leq r \text{ and } \W_i=0 \text{ otherwise, and}\\
 & \W_i^2\geq\W_{i-1}\W_{i+1},\; 1\leq i\leq n-1.
\end{split}
\end{equation}
In the rest of the proof we will construct two convex bodies $K,E$, with
$\dim K=s$, $\dim E=r$, $\dim(K+E)=n$ and $\W_i=\W_i(K;E)$, and so
$\sum_{i=0}^na_iz^i=\f{K}{E}(z)$. To this end we extend the construction
in \cite{She} to handle lower dimensional bodies as well and, as in
\cite{She}, the sets $K,E$ will be simplices.

Let $\e_i$, $1\leq i\leq n$, denote the $i$-th canonical unit vector, and
let $q_i=\alpha_i\e_i$, where $\alpha_i>0$ for $i=1,\dots,r$, $\alpha_i=0$
for $i=r+1,\dots,n$, and $\alpha_i\geq\alpha_{i+1}$ if
$i=n-s+1,\dots,r-1$. These numbers $\alpha_i$'s will be fixed at the end
of the proof. Let $K$, $E$ be the, respectively, $s$- and $r$-dimensional
simplices
\begin{equation}\label{e:def_S,T}
K=\conv\{0,\e_{n-s+1},\dots,\e_n\},\quad E=\conv\{0,q_1,\dots,q_r\}.
\end{equation}
Then $K+E=\conv\{0,\e_i,q_j,\e_i+q_j:n-s+1\leq i\leq n,1\leq j\leq r\}$,
but since $\alpha_j\geq\alpha_{j+1}$ for $j=n-s+1,\dots,r-1$, the points
$\e_i+q_j\in\conv\{0,\e_i+q_i,\e_j+q_j\}$ if $i<j$, and thus
\begin{equation}\label{e:S+T}
\begin{split}
K+E=\conv\{0,\e_i+q_j & :j\leq i,n-s+1\leq i\leq n,1\leq j\leq r,\\
             \e_i,q_j & :r+1\leq i\leq n,1\leq j\leq n-s\}.
\end{split}
\end{equation}
Now for $n-s+1\leq m\leq r+1$, let
\[
K_m=\conv\{0,\e_m,\dots,\e_n\},\quad E_m=\conv\{q_1,\dots,q_m\},
\]
with $K_{n+1}=\{0\}$, $E_{n+1}=\conv\{0,q_1,\dots,q_n\}$; notice that
$q_{r+1}=0$. In the following we will show by induction on the dimension
that $K+E$ is the interior-disjoint union of the sets $K_m+E_m$, i.e.,
\begin{equation}\label{e:disection}
K+E=\qquad\bigdotcup[4.7ex]_{m=n-s+1}^{r+1} (K_m+E_m),
\end{equation}
where $\dotcup$ denotes interior-disjoint union.

For $n=1$ the assertion is trivial. So let $n\geq 2$, and let
\[
\overline{K}=\conv\{0,\e_{n-s+1},\dots,\e_{n-1}\},\quad
\overline{E}=\left\{\begin{array}{ll}
 \conv\{0,q_1,\dots,q_r\} & \text{ if } r<n,\\[1mm]
 \conv\{0,q_1,\dots,q_{n-1}\} & \text{ if } r=n,
\end{array}\right.
\]
with $\overline{K}=\{0\}$ if $s=1$. Notice that in both cases,
$\dim\overline{K}+\dim\overline{E}=s-1+r\geq n-1$. Similarly as before we
consider, for $n-s+1\leq m\leq r+1$ (if $r<n$) or $n-s+1\leq m<n$ (if
$r=n$),
\[
\overline{K}_m=\conv\{0,\e_m,\dots,\e_{n-1}\},\quad
\overline{E}_m=\conv\{q_1,\dots,q_m\},
\]
where $\overline{K}_n=\{0\}$ and
$\overline{E}_n=\conv\{0,q_1,\dots,q_{n-1}\}$ (also for $m=r=n$). By
induction hypothesis,
\[
\overline{K}+\overline{E}=\left\{\begin{array}{ll}
\bigdotcup[1.3ex]_{m=n-s+1}^{r+1}(\overline{K}_m+\overline{E}_m) & \text{if } r<n\\[2mm]
\bigdotcup[1.3ex]_{m=n-s+1}^n(\overline{K}_m+\overline{E}_m) & \text{if }
r=n
\end{array}\right\}=:\qquad\bigdotcup[4.7ex]_{m=n-s+1}^{r+1,n}(\overline{K}_m+\overline{E}_m),
\]
and taking the orthogonal projection $\pi_n$ onto the coordinate
hyperplane $\e_n=0$ and the restriction $\pi:=(\pi_n)|_{K+E}$, we get
\[
K+E=\pi^{-1}\bigl(\overline{K}+\overline{E}\bigr)=
\qquad\bigdotcup[4.7ex]_{m=n-s+1}^{r+1,n}\pi^{-1}\bigl(\overline{K}_m+\overline{E}_m\bigr).
\]
It is easy to see that
$\pi^{-1}\bigl(\overline{K}_m+\overline{E}_m\bigr)=K_m+E_m$ for
$m=n-s+1,\dots,r+1$ when $r<n$ and $m=n-s+1,\dots,n-1$ when $r=n$. So we
get the required union for $K+E$ in $r+s-n+1$ interior-disjoint parts (cf.
\eqref{e:disection}) when $r<n$. Finally, if $r=n$,
\[
\pi^{-1}\bigl(\overline{K}_n+\overline{E}_n\bigr)=\conv\{0,q_j,q_j+\e_n:
1\leq j\leq n\}=(K_{n+1}+E_{n+1})\dotcup(K_n+E_n),
\]
providing the $s+1$ interior-disjoint parts in \eqref{e:disection} when
$r=n$.

Next, based on relation \eqref{e:disection}, we can compute the volume of
the polytope $K+E$. Since $(\aff K_m)\cap(\aff E_m)=\{q_m\}$ we get, for
all $n-s+1\leq m\leq r+1$ ($m\neq n+1$), that
\[
\begin{split}
\vol(K_m+E_m) & =\vol\left(K_m+\bigl(E_m|(\lin K_m)^{\bot}\bigr)\right)\\
 & =\vol_{n-m+1}(K_m)\vol_{m-1}\bigl(\conv\{0,q_1,\dots,q_{m-1}\}\bigr)\\
 & =\frac{1}{(n-m+1)!}\frac{\alpha_1\dots\alpha_{m-1}}{(m-1)!}
    =\frac{1}{n!}\binom{n}{m-1}\alpha_1\dots\alpha_{m-1};
\end{split}
\]
here we use $\vol_i$ to denote the $i$-dimensional volume in $\R^i$,
$L^{\bot}$ for the orthogonal complement of a linear subspace $L$ and
$M|L$ for the orthogonal projection of $M\varsubsetneq\R^n$ onto $L$.
Observe that if $r=n$ then
$\vol(K_{n+1}+E_{n+1})=\vol(E_{n+1})=(1/n!)\,\alpha_1\dots\alpha_n$. Thus,
by \eqref{e:disection},
\[
\vol(K+E)=\sum_{m=n-s+1}^{r+1}\vol(K_m+E_m)=\sum_{i=n-s}^r\binom{n}{i}\frac{1}{n!}\,\alpha_1\dots\alpha_i,
\]
where, if $s=n$, the first summand ($i=0$) is just $1/n!$. This says that
$\W_i(K;E)=(1/n!)\,\alpha_1\dots\alpha_i$ for $n-s\leq i\leq r$, and
$\W_i(K;E)=0$ otherwise.

Now we go back to our given sequence of real numbers $\W_0,\dots,\W_n\geq
0$ satisfying \eqref{e:cond_Wi}. Let
\[
\alpha_i=\left\{\begin{array}{ll}
 (n!\W_{n-s})^{1/(n-s)} & \text{for }\; i=1,\dots,n-s,\\
 \W_i/\W_{i-1} & \text{for }\; i=n-s+1,\dots,r,\\
 0 & \text{for }\; i=r+1,\dots,n.
\end{array}\right.
\]
Since $\W_i^2\geq\W_{i-1}\W_{i+1}$ we have $\alpha_i\geq\alpha_{i+1}$ for
$n-s+1\leq i\leq r$, and taking $K,E$ as defined in \eqref{e:def_S,T} we
get, for all $i=n-s,\dots,r$,
\[
\W_i(K;E)=\frac{1}{n!}\,\alpha_1\dots\alpha_{n-s}\alpha_{n-s+1}\dots\alpha_i
=\frac{1}{n!}(n!\W_{n-s})\frac{\W_i}{\W_{n-s}}=\W_i,
\]
and $\W_i=0$ otherwise.
\end{proof}

For complex numbers $z_1,\dots,z_r\in\C$ let
\begin{equation*}
\sy{i}{z_1,\dots,z_r}=\sum_{\substack{J\subseteq\{1,\dots,r\}\\\#J=i}}
\prod_{j\in J}z_j
\end{equation*}
denote the $i$-th elementary symmetric function of $z_1,\dots,z_r$, $1\leq
i\leq r$. In addition we set $\sy{0}{z_1,\dots,z_r}=1$. Using this
notation the following corollary is an immediate consequence of Lemma
\ref{l:charact_Steiner}.

\begin{corollary}\label{c:ultra-log_roots}
The complex numbers $\gamma_1,\dots,\gamma_r\in\C$ are the roots of a
Steiner polynomial $\f{K}{E}(z)$ of degree $r\leq n$, with $K,E\in\K^n$,
$\dim E=r$, $\dim K=s$, $\dim(K+E)=n$, if and only if
\begin{equation}\label{eq:condition}
\begin{split}
\text{\rm i)}\quad (-1)^i & \sy{i}{\gamma_1,\dots,\gamma_r}>0, \quad 0\leq i\leq r+s-n,\\
 & \sy{i}{\gamma_1,\dots,\gamma_r}=0, \quad r+s-n+1\leq i\leq r,\\
\text{\rm ii)}\;\; c_{r-i,n}\, & \sy{i}{\gamma_1,\dots,\gamma_r}^2\geq
\sy{i-1}{\gamma_1,\dots,\gamma_r}\sy{i+1}{\gamma_1,\dots,\gamma_r},\;
1\leq i\leq r-1.
\end{split}
\end{equation}
\end{corollary}

We conclude this section by three immediate applications  of Lemma
\ref{l:charact_Steiner}.
\begin{proposition}\label{prop:truncated_steiner}
All truncated binomial polynomials
$P_{j,k}^n(z)=\sum_{i=j}^k\binom{n}{i}z^i$, $0\leq j<k\leq n$, are Steiner
polynomials of convex bodies $K,E\in\K^n$ with $\dim K=n-j$, $\dim E=k$
and $\dim(K+E)=n$.
\end{proposition}
Hence in the following we consider $P_{j,k}^n(z)$ as Steiner polynomials.
In fact, by the proof of Lemma \ref{l:charact_Steiner}, $P_{j,k}^n(z)$ can
be realized as the Steiner polynomial $\f{K}{E}(z)$ of the bodies
$K=\conv\{0,\e_{j+1},\dots,\e_n\}$ and $E=\conv\{0,c\,\e_1,\allowbreak
\dots,c\,\e_j,\e_{j+1},\dots,\e_k\}$ with $c=(n!)^{1/j}$.

Second consequence deals with the derivative and antiderivative of Steiner
polynomials.
\begin{proposition}
Let $\f{K}{E}(z)=\sum_{i=0}^n a_i\,z^i$ be the Steiner polynomial of two
convex bodies $K,E\in\K^n$, $\dim(K+E)=n$. Then both, its derivative as
well as its antiderivative
\begin{equation*}
\f{K}{E}'(z)=\sum_{i=0}^{n-1} (i+1)\,a_{i+1}\,z^i\; \text{ and }\;
\int\f{K}{E}(z)\,{\rm d}z=\sum_{i=1}^{n+1}\frac{a_{i-1}}{i}\,z^i
\end{equation*}
are Steiner polynomials of appropriate convex bodies in $\K^{n-1}$ and
$\K^{n+1}$, respectively.
\end{proposition}
If $\dim K=n$, we may also add any constant term $c$ to the antiderivative
as long as $c\leq na_0^2/\bigl((n+1)a_1\bigr)$.

The last consequence regards Steiner polynomials with only real roots.
\begin{proposition}
For any given $n$ real numbers $\gamma_i\leq 0$, $i=1,\dots,n$, there
exist $K,E\in\K^n$ such that $\f{K}{E}(\gamma_i)=0$ for all $i=1,\dots,n$.
\end{proposition}
This is, for instance, due to the fact that the elementary symmetric
functions form an ultra-logconcave sequence (Newton inequalities, see e.g.
\cite{HLP}),
\[
\left(\frac{\sy{i}{\gamma_1,\dots,\gamma_n}}{\binom{n}{i}}\right)^2\geq
\frac{\sy{i-1}{\gamma_1,\dots,\gamma_n}}{\binom{n}{i-1}}\frac{\sy{i+1}{\gamma_1,\dots,\gamma_n}}{\binom{n}{i+1}},
\]
and so Lemma \ref{l:charact_Steiner} gives the result. In the case $n=2$
this means that given any pair $\gamma,\gamma'\in\r(2)$, we can find a
Steiner polynomial having these two roots. This property is, however, not
true in higher dimension if we also allow complex (non-real) numbers to be
involved. Indeed, in \cite[pp.~160-161]{HHC2} it is shown that if
$-a+b\im\in\r(3)$, then $-a+b\im,-a-b\im,-c$ are the roots of a Steiner
polynomial if and only if either $c\leq a-\sqrt{3}\,b$ or
$c\geq(a^2+b^2)/(a-\sqrt{3}\,b)$.

\section{On the boundary of the cones $\r(n)$}\label{s:boundary}

We start showing that all cones $\r(n)$  are closed.

\begin{proof}[Proof of Theorem \ref{t:R_n_closed}]
Let $\gamma\in\bd\r(n)$. Since we already know that the non-positive real
axis is always contained in $\r(n)$, we assume that $\gamma\not\in\R$. Let
$(\gamma_j)_{j\in\N}\varsubsetneq\inter\r(n)$ be a sequence of complex
numbers converging to $\gamma$. For each $j\in\N$, since
$\gamma_j\in\inter\r(n)$, there exists a pair of convex bodies
$(K_j,E_j)\in\K^n\times\K^n$, $\dim(K_j+E_j)=n$, such that
$\f{K_j}{E_j}(\gamma_j)=0$.

Notice that we can always choose $K_j,E_j$ such that $\vol(K_j+E_j)=1$.
Otherwise, since $\vol(K_j+E_j)>0$, it suffices to consider the new convex
bodies $K_j'=1/\vol(K_j+E_j)^{1/n}K_j$ and
$E_j'=1/\vol(K_j+E_j)^{1/n}E_j$, for which it clearly holds
$\f{K_j'}{E_j'}(\gamma_j)=\bigl(1/\vol(K_j+E_j)\bigl)\f{K_j}{E_j}(\gamma_j)=0$,
and moreover,
\[
\vol(K_j'+E_j')=\f{K_j'}{E_j'}(1)=\frac{1}{\vol(K_j+E_j)}\f{K_j}{E_j}(1)=1.
\]
Observe that since
$\vol(K_j+E_j)=\sum_{i=0}^n\binom{n}{i}\W_i(K_j;E_j)=1$, all
quermassintegrals $\W_i(K_j;E_j)\in[0,1]$, $i=0,\dots,n$, and not all of
them are zero. Then, denoting by $\W_{i,j}=\W_i(K_j;E_j)$, we can assure
that the bounded sequence of ($n+1$)-tuples of numbers
$(\W_{0,j},\dots,\W_{n,j})_{j\in\N}$ has a convergent subsequence to an
($n+1$)-tuple $(\W_0,\dots,\W_n)$, and without loss of generality we
assume that $(\W_{0,j},\dots,\W_{n,j})_{j\in\N}$ is the convergent
subsequence.

By continuity, the numbers $\W_0,\dots,\W_n$ also satisfy inequalities
\eqref{e:special_af2}, and thus the sequence
$\bigl\{a_i=\binom{n}{i}\W_i:i=0,\dots,n\bigr\}$ is ultra-logconcave.
Moreover,
\[
\sum_{i=0}^n\binom{n}{i}\W_i=\lim_{j\rightarrow\infty}\sum_{i=0}^n\binom{n}{i}\W_{i,j}
=\lim_{j\rightarrow\infty}\vol(K_j+E_j)=1,
\]
i.e., the polynomial
$\sum_{i=0}^n\binom{n}{i}\W_iz^i=\sum_{i=0}^na_iz^i\neq 0$. Therefore, the
property $a_i>0$ for all $n-s\leq i\leq r$ and $a_i=0$ otherwise, holds
for suitable $r,s\in\{1,\dots,n\}$. Then Lemma \ref{l:charact_Steiner}
ensures that $\sum_{i=0}^n\binom{n}{i}\W_iz^i$ is a Steiner polynomial of
two convex bodies $K,E\in\K^n$ with $\dim K=s$, $\dim E=r$. By continuity,
since $\f{K_j}{E_j}(\gamma_j)=0$ for all $j\in\N$ and the sequence of
complex numbers $(\gamma_j)_{j\in\N}$ converges to $\gamma$, we have
$\f{K}{E}(\gamma)=0$, i.e., $\gamma\in\r(n)$. This shows that the cone
$\r(n)$ is closed.
\end{proof}

Since $\r(n)$ is closed, we may ask which pairs of convex bodies or
Steiner polynomials determine the boundary $\bd\r(n)\backslash\R_{\leq
0}$. We recall (cf. Proposition~\ref{p:cone3}) that if $E\in\K^3_0$ is a
cap-body of a planar convex body $K$, then $(K,E)$ is a boundary-pair. We
also notice that if $E\in\K^4_0$ is a cap-body of $K$ with $\dim K=3$,
then the condition for the boundary in Proposition \ref{p:r(4)} is
satisfied, i.e., $(K,E)$ is also a boundary-pair in dimension $4$. However
this is not the case for $n\geq 5$: in general, if $K\in\K^n$ with $\dim
K=n-1$ and $E\in\K^n_0$ is a cap-body of $K$, then
$\vol(E)=\W_0(E;K)=\dots=\W_{n-1}(E;K)\neq 0$ (see \cite[proof of
Theorem~6.6.16, p.~368]{Sch}); so, since $\W_0(K;E)=0$ we get
\[
\f{K}{E}(z)=\sum_{i=1}^n\binom{n}{i}\W_i(K;E)z^i=\sum_{i=1}^n\binom{n}{i}\W_{n-i}(E;K)z^i
=\vol(E)P_{1,n}^n(z).
\]
Then it can be checked that all roots of the Steiner polynomial
$P_{1,5}^5(z)$ lie in the interior of the cone determined by the complex
number $-0.5000+0.8660\im$, which is a root of the Steiner polynomial
$P_{1,4}^5(z)$ (cf. Table~\ref{ta:computations}). Analogously for
dimensions $n=6,7,8,9$. Finally, it can be easily seen (cf. also
\cite[Corollary 3.1]{HHC2}) that all roots of $P_{1,n}^n(z)$ have
non-positive real part, and thus, because of the non-stability of the
Steiner polynomial for $n\geq 10$ (Proposition \ref{p:cones_10_11}), they
cannot determine the boundary.


\begin{remark}\label{r:computations}
Numerical computations suggest that for each $n$ and suitable $0<j<k\leq
n$, the Steiner polynomials
\begin{equation*}\label{e:binomial}
P_{j,k}^n(z)=\sum_{i=j}^k\binom{n}{i}z^i
\end{equation*}
have a root on the boundary $\bd\r(n)\backslash\R_{\leq 0}$ (cf.
Conjecture \ref{co:binomial}). Table \ref{ta:computations} lists, for
$n\leq 20$, the indices $j$ and $k$ of those Steiner polynomials
$P_{j,k}^n(z)$ having a root $\gamma$ of minimal angle $\alpha$ with the
positive real axis.
\end{remark}

\begin{table}[htb]
\def\arraystretch{1.1}
\begin{tabular}{llll}
$n=3$ & $j=1$, $k=3$ & $\gamma=-1.5000+0.8660\im$ & $\alpha=2.6179$\\
$n=4$ & $j=1$, $k=4$ & $\gamma=-1.0000+1.0000\im$ & $\alpha=2.3561$\\
$n=5$ & $j=1$, $k=4$ & $\gamma=-0.5000+0.8660\im$ & $\alpha=2.0943$\\
$n=6$ & $j=1$, $k=5$ & $\gamma=-0.3856+0.9226\im$ & $\alpha=1.9667$\\
$n=7$ & $j=2$, $k=6$ & $\gamma=-0.3249+1.2279\im$ & $\alpha=1.8294$\\
$n=8$ & $j=2$, $k=6$ & $\gamma=-0.1464+0.9892\im$ & $\alpha=1.7177$\\
$n=9$ & $j=2$, $k=7$ & $\gamma=-0.0698+0.9975\im$ & $\alpha=1.6406$\\
$n=10$ & $j=3$, $k=8$ & $\gamma=\hphantom{-}0.0158+1.1903\im$ & $\alpha=1.5574$\\
$n=11$ & $j=3$, $k=8$ & $\gamma=\hphantom{-}0.0854+0.9963\im$ & $\alpha=1.4852$\\
$n=12$ & $j=4$, $k=9$ & $\gamma=\hphantom{-}0.1533+1.1549\im$ & $\alpha=1.4388$\\
$n=13$ & $j=4$, $k=10$ & $\gamma=\hphantom{-}0.2127+1.1256\im$ & $\alpha=1.3840$\\
$n=14$ & $j=4$, $k=10$ & $\gamma=\hphantom{-}0.2400+0.9707\im$ & $\alpha=1.3284$\\
$n=15$ & $j=5$, $k=11$ & $\gamma=\hphantom{-}0.3139+1.0864\im$ & $\alpha=1.2895$\\
$n=16$ & $j=5$, $k=11$ & $\gamma=\hphantom{-}0.3121+0.9500\im$ & $\alpha=1.2533$\\
$n=17$ & $j=5$, $k=12$ & $\gamma=\hphantom{-}0.3452+0.9384\im$ & $\alpha=1.2182$\\
$n=18$ & $j=6$, $k=13$ & $\gamma=\hphantom{-}0.4186+1.0258\im$ & $\alpha=1.1833$\\
$n=19$ & $j=6$, $k=13$ & $\gamma=\hphantom{-}0.4076+0.9131\im$ & $\alpha=1.1509$\\
$n=20$ & $j=7$, $k=14$ & $\gamma=\hphantom{-}0.4727+0.9917\im$ & $\alpha=1.1259$\\
\end{tabular}

\medskip \caption{Numerical computations for $\bd\r(n)\backslash\R_{\leq 0}$, $n\leq
20$.}\label{ta:computations}
\end{table}

Of particular interest is  also the entry for dimension $n=10$ in Table
\ref{ta:computations}. Here we have for the first time a root $\gamma$
with positive real part and so $P_{3,8}^{10}(z)$ is a non-weakly stable
Steiner polynomial. Together with known results this settles the question
when Steiner polynomials are (weakly) stable.
\begin{proof}[Proof of Proposition \ref{p:cones_10_11}]
The (weak) stability of the Steiner polynomial was shown for all
dimensions $n\leq 9$ in \cite[Proposition~1.1]{HHC2}, as well as its
non-stability when $n\geq 12$ \cite[Remark~3.2]{HHC}. Thus just the cases
$n=10,11$ remain to be considered, but Table \ref{ta:computations}
provides two non-weakly stable Steiner polynomials in these dimensions.
\end{proof}

\section{On the monotonicity of the cones $\r(n)$}\label{s:monotonicity}

First we observe that it is easy to see that $\r(n)\subseteq\r(n+1)$. To
this end, let $\gamma\in\r(n)$ and $K,E\in\K^n$  such that
$\f{K}{E}(\gamma)=0$. Identifying $K$ and $E$ with their canonical
embedding in the hyperplane $\{\e_{n+1}\}^{\bot}\varsubsetneq\R^{n+1}$,
let $E'=E\times\conv\{0,\e_{n+1}\}$ be the prism over $E$ of height 1 in
the direction $\e_{n+1}$. Then we observe that
\begin{equation*}
\vol(K+\lambda\,E^\prime)=\vol\bigl((K+\lambda
E)\times\lambda\conv\{0,\e_{n+1}\}\bigr)=\lambda\,\vol_n(K+\lambda E),
\end{equation*}
i.e., $\f{K}{E'}(z)=z\f{K}{E}(z)$ and thus $\f{K}{E'}(\gamma)=0$. Hence
$\gamma\in\r(n+1)$, which shows that $\r(n)\subseteq\r(n+1)$. Theorem
\ref{t:R_n_inclusion} states that this inclusion is strict.

\begin{proof}[Proof of Theorem \ref{t:R_n_inclusion}]
Let $\gamma_1\in\bd\r(n)\backslash\R_{\leq 0}$. By Theorem
\ref{t:R_n_closed} $\r(n)$ is closed, and hence $\gamma_1$ is a root of
some Steiner polynomial $\f{K}{E}(z)$ of degree $r\leq n$, with
$K,E\in\K^n$, $\dim E=r$, $\dim K=s$, $\dim(K+E)=n$. Let
$\gamma_2,\dots,\gamma_r$ be the remaining roots of the polynomial, where
$\gamma_2=\overline{\gamma}_1$ is the complex conjugate of $\gamma_1$. We
may assume that $\gamma_1,\dots,\gamma_{r+s-n}\neq 0$ and
$\gamma_{r+s-n+1}=\dots=\gamma_r=0$. So, $0$ is (exactly) an ($n-s$)-fold
root.

In the following we will show that $\gamma_1$ lies in the interior of
$\r(n+1)$, i.e., we will prove the existence of $\varepsilon_0>0$ such
that for any $z\in\C$ with modulus $|z|=1$, the $r+1$ complex numbers
$\rho_1=\gamma_1+\varepsilon_0z,\rho_2=\gamma_2+\varepsilon_0\overline{z},\gamma_3,\dots,\gamma_r,0$
are the roots of a Steiner polynomial $\f{K'}{E'}(z)$ of degree $r+1$ with
$K',E'\in\K^{n+1}$, $\dim E'=r+1$, $\dim K'=s$ and $\dim(K'+E')=n+1$.
According to Corollary~\ref{c:ultra-log_roots} this is equivalent to show
that
\begin{equation*}\label{eq:condition_r+1}
\begin{split}
\text{\rm I)} & \quad
(-1)^i\sy{i}{\rho_1,\rho_2,\gamma_3\dots,\gamma_r,0}>0,\quad 0\leq i\leq r+s-n, \\
 & \hspace*{1.3cm}\sy{i}{\rho_1,\rho_2,\gamma_3\dots,\gamma_r,0}=0,\quad r+s-n+1\leq i\leq r+1,\\
\text{\rm II)} & \quad c_{r+1-i,n+1}\,\sy{i}{\rho_1,\rho_2,\gamma_3\dots,\gamma_r,0}^2\\
 & \hspace*{1.62cm}\geq\sy{i-1}{\rho_1,\rho_2,\gamma_3\dots,\gamma_r,0}\sy{i+1}{\rho_1,\rho_2,\gamma_3\dots,\gamma_r,0},\;
1\leq i\leq r.
\end{split}
\end{equation*}
To this end we note that, for $0\leq i\leq r$ and for any $\varepsilon>0$,
\begin{equation}\label{eq:extension}
\begin{split}
\sy{i}{\gamma_1\!+\varepsilon z,\gamma_2+\varepsilon\overline{z},
    \gamma_3,\dots,\gamma_r,0} & =\sy{i}{\gamma_1\!+\varepsilon z,\gamma_2+\varepsilon\overline{z},
    \gamma_3,\dots,\gamma_r},\\
\sy{r+1}{\gamma_1\!+\varepsilon z,\gamma_2+\varepsilon\overline{z},
    \gamma_3,\dots,\gamma_r,0} & =0.
\end{split}
\end{equation}
Since $n+1-s$ of the $r+1$ numbers $\gamma_1+\varepsilon
z,\gamma_2+\varepsilon\overline{z},\gamma_3,\dots,\gamma_r,0$ are zero, we
also have that, for any $\varepsilon>0$,
\begin{equation}\label{eq:zeronew}
\sy{i}{\gamma_1+\varepsilon z,
\gamma_2+\varepsilon\overline{z},\gamma_3,\dots,\gamma_r,0}=0\quad \text{
for }\; i\geq r+s-n+1.
\end{equation}
Obviously, the numbers $\gamma_1+\varepsilon z,\gamma_2+ \varepsilon
\overline{z},\gamma_3,\dots,\gamma_r,0$ are roots of a polynomial with
real coefficients. Hence, in view of \eqref{eq:extension},
\eqref{eq:condition} i) and the continuity of polynomials, there exists
$\varepsilon_1>0$ such that for any $0<\varepsilon\leq\varepsilon_1$
\begin{equation*}
\begin{split}
(-1)^i\, & \sy{i}{\gamma_1+\varepsilon z,\gamma_2+\varepsilon\overline{z},\gamma_3,\dots,\gamma_r,0}\\
 & =(-1)^i\,\sy{i}{\gamma_1+\varepsilon z,\gamma_2+\varepsilon\overline{z},\gamma_3,\dots,\gamma_r}>0,
    \quad 0\leq i\leq r+s-n.
\end{split}
\end{equation*}
So, with \eqref{eq:zeronew} both conditions in I) are satisfied for
$\varepsilon\leq\varepsilon_1$.

Relation \eqref{eq:zeronew} also implies that the inequalities in II) are
certainly satisfied for $r+s-n\leq i\leq r$. So it remains to consider
$1\leq i<r+s-n$. By \eqref{eq:condition}~ii) we know that
\begin{equation*}
c_{r-i,n}\,\sy{i}{\gamma_1,\dots,\gamma_r}^2\geq\sy{i-1}{\gamma_1,\dots,\gamma_r}\sy{i+1}{\gamma_1,\dots,\gamma_r},
\end{equation*}
and since $c_{r+1-i,n+1}>c_{r-i,n}$ for all $1\leq i\leq r$ and
$\sy{i}{\gamma_1,\dots,\gamma_r}^2>0$ for $0\leq i\leq r+s-n$
(cf.~\eqref{eq:condition}~i)), we get that
\begin{equation*}
c_{r+1-i,n+1}\,\sy{i}{\gamma_1,\dots,\gamma_r}^2>\sy{i-1}{\gamma_1,\dots,\gamma_r}\sy{i+1}{\gamma_1,\dots,\gamma_r}
\end{equation*}
for all $1\leq i<r+s-n$. Hence, as before, by continuity of polynomials,
there exists $\varepsilon_2>0$ such that
\begin{equation*}
\begin{split}
c_{r+1-i,n+1}\, & \sy{i}{\gamma_1+\varepsilon z,\gamma_2+\varepsilon\overline{z},\gamma_3,\dots,\gamma_r}^2\\
 >\, & \sy{i-1}{\gamma_1+\varepsilon z,\gamma_2+\varepsilon\overline{z},\gamma_3,\dots,\gamma_r}
    \sy{i+1}{\gamma_1+\varepsilon z,\gamma_2+\varepsilon\overline{z},\gamma_3,\dots,\gamma_r}
\end{split}
\end{equation*}
for all $0<\varepsilon\leq\varepsilon_2$ and $1\leq i<r+s-n$. On account
of \eqref{eq:extension} we obtain II) for $\varepsilon\leq\varepsilon_2$,
and the assertion follows with
$\varepsilon_0=\min\{\varepsilon_1,\varepsilon_2\}$.
\end{proof}

As a corollary of the above proof we obtain a necessary condition
for convex bodies forming a boundary-pair.

\begin{proof}[Proof of Corollary \ref{c:A-F_ineq}]
For $\gamma\in\bd\r(n)\backslash\R_{\leq 0}$, $n\geq3$, let $K,E\in\K^n$
be such that $\f{K}{E}(\gamma)=0$, and let
$\overline{\gamma},\gamma_3,\dots,\gamma_n$ be the remaining roots of
$\f{K}{E}(z)$.

If we assume that $K,E$ are not extremal sets in any Aleksandrov-Fenchel
inequality, i.e., if we have strict inequalities in \eqref{e:special_af2},
then for all $1\leq i\leq n-1$ we get by Corollary~\ref{c:ultra-log_roots}
\[
c_{r-i,n}\,\sy{i}{\gamma,\overline{\gamma},\gamma_3,\dots,\gamma_n}^2>
\sy{i-1}{\gamma,\overline{\gamma},\gamma_3,\dots,\gamma_n}\sy{i+1}{\gamma,\overline{\gamma},\gamma_3,\dots,\gamma_n}.
\]
By the continuity of the elementary symmetric functions, for
$\varepsilon>0$ small enough, the numbers $\gamma+\varepsilon
z,\overline{\gamma}+\varepsilon\overline{z}, \gamma_3,\dots,\gamma_n$ are
roots of a polynomial with real coefficients, satisfying also conditions
i) and ii) of Corollary \ref{c:ultra-log_roots} for any $z\in\C$ with
$|z|=1$. This implies that $\{\gamma+\varepsilon
z:|z|=1\}\varsubsetneq\r(n)$, contradicting that
$\gamma\in\bd\r(n)\backslash\R_{\leq 0}$.
\end{proof}

We conclude this section studying the behavior of the cones for high
dimensions, i.e., we prove Theorem \ref{t:n->infty}.

\begin{proof}[Proof of Theorem \ref{t:n->infty}]
The proof is based on known results on the distribution of the roots of
the truncated binomial polynomials
$P_{0,k}^n(z)=\sum_{i=0}^k\binom{n}{i}z^i$, $0<k\leq n$, which are also
Steiner polynomials (cf.~Proposition \ref{prop:truncated_steiner}).

Let $\{k_n:n\in\N\}$ be any sequence of positive integer numbers such that
$\alpha=\lim_{n\to\infty}k_n/n\in(0,1)$. By \cite[Remark 1]{O04} we have
that the set of accumulation points of
$\bigcup_{n=1}^{\infty}\bigl\{z\in\C:P_{0,{k_n}}^n(z)=0\bigr\}$ coincides
with the set
\begin{equation*}
\left\{z\in\C:|z|=\alpha\,(1-\alpha)^{1/\alpha-1}\,|1+z|^{1/\alpha} \text{
and }
\left|z-\frac{\alpha^2}{1-\alpha^2}\right|\leq\frac{\alpha}{1-\alpha^2}\right\}.
\end{equation*}
Hence, taking $k_n=\lfloor n/2\rfloor$, it can be checked that $1$ is
contained in the above set of accumulation points, and so we know that
there exists a sequence $\gamma_n\in\C^+\setminus\R_{>0}$, $n\in\N$, such
that for each $n\in\N$ there is $m_n\in\N$ with
\begin{equation}\label{eq:roots_binom}
\lim_{n\to\infty}\gamma_n=1\quad\text{ and }\quad P_{0,\lfloor
m_n/2\rfloor}^{m_n}(\gamma_n)=0.
\end{equation}

Now let $\gamma\in\C^+\setminus\R_{>0}$. By the choice of the sequence
$\gamma_n$ (cf.~\eqref{eq:roots_binom}) we can find an $n_{\gamma}\in \N$
such that $\gamma$ is contained in the interior of the cone generated by
the negative $x$-axis and $\gamma_{n_{\gamma}}$, which in particular
implies, by the convexity of the cone $\r(n_{\gamma})$ (cf.~Theorem
\ref{t:cone}), that $\gamma\in\r(n_{\gamma})$. By Theorem
\ref{t:R_n_inclusion} we get the desired statement.
\end{proof}

\section{The $4$-dimensional cone}\label{s:apendix}

We conclude the paper by characterizing the cone of roots of $4$-dimensional
Steiner polynomials, for which it suffices to determine its boundary
(cf.~Theorem~\ref{t:cone}).

\begin{proof}[Proof of Proposition \ref{p:r(4)}]
First we notice that $\{x+y\im\in\C^+:x+y\leq 0\}\subseteq\r(4)$. Indeed,
since $\r(4)$ is a convex cone containing $\R_{\leq 0}$
(Theorem~\ref{t:cone}), it suffices to prove that $-1+\im\in\r(4)$, which
follows from the fact that $P_{1,4}^4(-1+\im)=0$ and
Proposition~\ref{prop:truncated_steiner}.

Next we determine conditions verified by a pair of convex bodies whose
Steiner polynomial has $-1+\im$ as a root. We have to distinguish two
cases. If $E\in\K^4_0$ then such a polynomial has to take the form
\[
\f{K}{E}(z)=\sum_{i=0}^4\binom{4}{i}\W_i(K;E)z^i=\W_4(K;E)(z^2+2z+2)(z^2+cz+d),
\]
for certain $c,d\geq 0$ because it is weakly stable (cf. Proposition
\ref{p:cones_10_11}). Then we have the identities
\begin{equation}\label{e:W_i_cd}
\begin{split}
2+c=4\frac{\W_3(K;E)}{\W_4(K;E)}, & \quad 2(c+1)+d=6\frac{\W_2(K;E)}{\W_4(K;E)},\\
2(c+d)=4\frac{\W_1(K;E)}{\W_4(K;E)}, & \quad
2d=\frac{\W_0(K;E)}{\W_4(K;E)}.
\end{split}
\end{equation}
Inequalities \eqref{e:special_af2} for $i=3$, $i=2$ and $i=1$ yield, in
terms of $c,d$, respectively,
\begin{equation*}
\begin{split}
3c^2-4c-8d-4 & \geq 0,\\
c^2+(d+2)c-2(d^2-5d+4) & \leq 0,\\
3c^2-2dc-d^2-8d & \geq 0,
\end{split}
\end{equation*}
which, since $c,d\geq 0$, are equivalent to
\begin{equation*}
\begin{split}
c & \geq\frac{2}{3}\left(1+\sqrt{2}\sqrt{2+3d}\right),\\
c & \leq d-4\;\text{ if }\,d\geq 2\quad\text{ and }\quad c\leq 2(1-d)\;\text{ if }\,d\leq 2,\\
c & \geq\frac{1}{3}\left(d+2\sqrt{d(d+6)}\right),
\end{split}
\end{equation*}
respectively. A straightforward computation allows to conclude that the
three above inequalities hold simultaneously if and only if $d=0$ and
$c=2$. Then,
$\f{K}{E}(z)=\W_4(K;E)(z^4+4z^3+6z^2+4z)=\W_4(K;E)P_{1,4}^4(z)$. In
particular, $\W_0(K;E)=0$ (cf. \eqref{e:W_i_cd}) and, in view of
$\W_1(K;E)>0$, this shows that $\dim K=3$ and moreover, it holds
$\W_1(K;E)=\W_2(K;E)=\W_3(K;E)=\W_4(K;E)$.

Now we suppose $\dim E<4$. Then the polynomial has to take the form
\[
\f{K}{E}(z)=(z^2+2z+2)(cz+d)=cz^3+(d+2c)z^2+2(c+d)z+2d,
\]
for certain $c,d\geq 0$ and applying Lemma \ref{l:charact_Steiner} it is
easy to check that it is a Steiner polynomial if and only if $c=d$. Notice
that it cannot be $c=d=0$. Hence $\f{K}{E}(z)=cz^3+3cz^2+4cz+2c$, implying
that
\[
\frac{1}{2}\W_0(K;E)=\W_1(K;E)=2\W_2(K;E)=4\W_3(K;E)=c\neq 0
\]
and, in particular, that $\dim K=4$. In both cases we get the required
equalities $\W_i(K;E)^2=\W_{i-1}(K;E)\W_{i+1}(K;E)$, for $i=2,3$.

Finally we prove that $\r(4)=\{x+y\im\in\C^+:x+y\leq 0\}$. Thus we assume
that $\gamma=-1+(1+\varepsilon)\im\in\r(4)$ for $\varepsilon>0$, i.e.,
that there exist $K,E\in\K^n$ such that $\f{K}{E}(\gamma)=0$, and we will
get a contradiction. Then (see \cite[Lemma~2.1]{HHC2})
$\gamma-\varepsilon$ is a root of $\f{K+\varepsilon E}{E}(z)$. But since
$\gamma-\varepsilon=-(1+\varepsilon)+(1+\varepsilon)\im$, the previous
property implies that either $\dim(K+\varepsilon E)=3$ with $E\in\K^4_0$,
which is clearly not possible, or $\dim E=3$ and $\vol(K+\varepsilon
E)=\W_i(K+\varepsilon E;2E)\neq 0$, $i=1,2,3$, which also leads to a
contradiction. Indeed, if
\[
\W_0(K+\varepsilon E;2E)=\W_1(K+\varepsilon E;2E)= \W_2(K+\varepsilon
E;2E)=\W_3(K+\varepsilon E;2E),
\]
we find by the Steiner formulae for quermassintegrals (see \cite[(5.1.27)
and p.~212]{Sch}) that
\[
\W_2(K;E)=2(1-\varepsilon)\W_3(K;E)\;\text{ and }\;
\W_1(K;E)=(4+3\varepsilon^2-6\varepsilon)\W_3(K;E).
\]
Notice that this implies $\varepsilon<1$. However, substitution of the
above expressions in inequality \eqref{e:special_af2} for $i=2$ leads to
$\varepsilon\geq 2$, a contradiction.
\end{proof}

\end{document}